\documentclass[article]{amsart}
\usepackage{amsmath}
\usepackage{mathtools}
\usepackage{amssymb}
\usepackage{amsthm}
\usepackage{graphicx}
\usepackage{enumerate}
\usepackage[mathscr]{eucal}
\usepackage{tikz}
\usetikzlibrary{decorations.text,calc,arrows.meta}
\theoremstyle{plain}

\numberwithin{equation}{section}
\usepackage[english]{babel}

\newtheorem{theorem}{Theorem}[section]
\newtheorem{cor}[theorem]{Corollary}
\newtheorem{lemma}[theorem]{Lemma}
\theoremstyle{definition}
\newtheorem*{acknowledgement}{\textnormal{\textbf{Acknowledgements}}}
\newtheorem{definition}[theorem]{Definition}
\newtheorem{remark}{Remark}[theorem]

\newtheorem{prop}[theorem]{Proposition}
\newtheorem{example}[theorem]{Example}

\usepackage{hyperref}
\hypersetup{
    colorlinks=false,
    linkcolor=blue,
    filecolor=magenta,      
    urlcolor=cyan,
}
\newcommand{\x}{\mathbf{x}}

\begin{document}

\title[Birkhoff-James orthogonality in complex Banach spaces]{Birkhoff-James orthogonality in complex Banach spaces and  Bhatia-\v{S}emrl Theorem revisited}

\author[Roy Bagchi Sain]{Saikat Roy, Satya Bagchi, Debmalya Sain}

\address[Roy]{Department of Mathematics\\ National Institute of Technology Durgapur\\ Durgapur 713209\\ West Bengal\\ INDIA\\}
\email{saikatroy.cu@gmail.com}

\address[Bagchi]{Department of Mathematics\\ National Institute of Technology Durgapur\\ Durgapur 713209\\ West Bengal\\ INDIA\\}
\email{satya.bagchi@maths.nitdgp.ac.in}

\address[Sain]{Department of Mathematics\\ Indian Institute of Science\\ Bengaluru 560012\\ Karnataka \\INDIA\\ }
\email{saindebmalya@gmail.com}

                \newcommand{\acr}{\newline\indent}


\maketitle

\begin{abstract}
We explore Birkhoff-James orthogonality of two elements in a complex Banach space by using the directional approach. Our investigation illustrates the geometric distinctions between a smooth point and a non-smooth point in a complex Banach space. As a concrete outcome of our study, we obtain a new proof of the Bhatia-\v{S}emrl Theorem on orthogonality of linear operators.
\end{abstract}

\section{Introduction}

The importance of Birkhoff-James orthogonality in the study of the geometry of Banach spaces is undeniable. Several mathematicians applied Birkhoff-James orthogonality techniques to investigate the geometry of Banach spaces from time to time; see \cite{BS,BP,K,PSMM,S,SP,SRBB}. In the present work, we study a comparatively weaker version of Birkhoff-James orthogonality in complex Banach spaces, namely, the directional orthogonality. The objective of the current article is twofold: we characterize smoothness of an element in complex Banach spaces and engage the directional orthogonality techniques to obtain a different proof of the classical Bhatia-\v{S}emrl Theorem.

\subsection{Notations and terminologies} Letters $ \mathbb{X},~ \mathbb{Y} $ denote Banach spaces and the symbol $ \mathbb{H} $ is reserved for a Hilbert space. Let $ S_{\mathbb{X}} $ denote the unit sphere of the Banach space $ \mathbb{X}, $ i.e., $ S_{\mathbb{X}} := \{ x \in \mathbb{X} :~ \|x\|=1 \}. $ The unit circle in $ \mathbb{C} $ is denoted by $ S^{1}, $ i.e., $ S^{1} := \{ \gamma\in \mathbb{C} :~ |\gamma|=1 \}.$ Let $\mu \in \mathbb{C}$ be non-zero. Then $\mu := |\mu|e^{i\theta}$, where $\theta$ denotes the argument of $\mu.$ Throughout the article we consider the interval $[0,2\pi)$ to be the domain of argument. Let $ \mathsf{Re}~\mu $, $ \mathsf{Im}~\mu $, and $ \mathsf{arg}~\mu $ denote as usual the real part of $ \mu $, the imaginary part of $ \mu $, and the argument of $\mu,$ respectively. Given any $z:= a+ib \in \mathbb{C},$ we denote the conjugate of $z$ by $\overline z:=a-ib.$\\

Unless otherwise stated, we consider Banach spaces, in particular Hilbert spaces, (mostly) over the field $\mathbb{C}$ of complex numbers.\\

Let $ \mathbb{L}(\mathbb{X},\mathbb{Y}) $ denote the Banach space of all bounded linear operators from $ \mathbb{X} $ to $ \mathbb{Y} $ endowed with the usual operator norm. We write $ \mathbb{L}(\mathbb{X},\mathbb{Y}) = \mathbb{L}(\mathbb{X}), $ if $ \mathbb{X}=\mathbb{Y}. $ For any $ T \in \mathbb{L}(\mathbb{X},\mathbb{Y}), $ we denote the norm attainment set of $ T $ by $ M_T $, i.e., $$ M_T:= \{ x \in S_{\mathbb{X}} :~ \|Tx\| = \| T \| \}.$$ Given any two elements $ x,y \in \mathbb{X}, $ we say that $ x $ is Birkhoff-James orthogonal \cite{B,J,Ja} to $ y, $ written as $ x \perp_B y, $ if $$ \|x+\lambda y\| \geq \|x\|,~\mathrm{~for~ all~ }~ \lambda \in \mathbb{C}. $$ Let $x\in \mathbb{X}$ be non-zero and let $\mathbb{J}(x)$ be a subset of $S_{\mathbb{X}^*},$ defined by
\begin{align*}
\mathbb{J}(x):=\{x^*\in S_{\mathbb{X}^*}:~ x^*(x)=\|x\|\}.
\end{align*}
It is well-known \cite[Page 268]{Ja} that in the real case, $x\perp_B y$ (if and) only if there exists $u^*\in \mathbb{J}(x)$ such that $u^*(y)=0.$ It comes easily from convex analysis that the same characterization holds true in the complex case as well. We say that $x$ is smooth if the collection $\mathbb{J}(x)$ is singleton. The Banach space $\mathbb{X}$ is called smooth, if each of its non-zero elements is smooth.\\

We would like to remark here that another study on complex Birkhoff-James orthogonality was conducted in \cite{PSMM}, although most of our results in the present article differ in spirit from \cite{PSMM}. However, we do use some of the notations from \cite{PSMM}, which we mention below.

\begin{definition}
Let $ \mathbb{X} $ be a complex Banach space and let $ \gamma\in S^1 $. For any two elements $ x, y \in \mathbb{X} $, $ x $ is said to be \emph{orthogonal} to $ y $ \emph{in the direction of $ \gamma $}, written as $ x\perp_\gamma y $, if $ \| x + t \gamma y \| \geq \| x \|, $ for all $ t\in \mathbb{R} $.
\end{definition}

\begin{definition}
Let $ \mathbb{X} $ be a complex Banach space and let $ \gamma \in S^1 $. For any two elements $ x,y\in \mathbb{X} $, we say that $ y $ lies in the \emph{positive part} of $ x $ in the direction of $ \gamma $, written as $ y \in (x)_{\gamma}^{+} $, if $ \| x + t \gamma y \| \geq \| x \|, $ for all $ t\geq 0 $. Similarly, we say that $ y $ lies in the \emph{negative part} of $ x $ in the direction of $ \gamma $, written as $ y \in (x)_{\gamma}^{-} $, if $ \| x + t \gamma y \| \geq \| x \|, $ for all $ t\leq 0 $.
\end{definition}

The Bhatia-\v{S}emrl Theorem \cite{BS} provides a nice characterization of Birkhoff-James orthogonality of linear operators on a finite-dimensional Hilbert space in terms of orthogonality of certain special vectors in the ground space.

\begin{theorem}(Bhatia-\v{S}emrl Theorem)\label{Bhatia-Semrl}
Let $ \mathbb{H} $ be a finite-dimensional Hilbert space. Let $ T, A \in \mathbb{L}(\mathbb{H}) $. Then $ T \perp_B A $ if and only if there exists $ x\in M_T $ such that $ \langle Tx, Ax \rangle = 0. $
\end{theorem} 

In view of this seminal result, the study of the Bhatia-\v{S}emrl type theorems in the setting of real Banach spaces has been conducted in \cite{S,SP}. Indeed, using Theorem 2.1 and Theorem 2.2 of \cite{SP}, an elementary proof of the Bhatia-\v{S}emrl Theorem can be obtained in the real setting.  In our present work, we study Birkhoff-James orthogonality in complex Banach spaces from a geometric point of view. As an outcome of our exploration, we furnish an elementary proof of the Bhatia-\v{S}emrl Theorem in the complex case. We note that the Toeplitz-Hausdorff Theorem was substantially used in the proof of the Bhatia-\v{S}emrl Theorem in each of \cite{BS,BP,TA}. We also apply the Toeplitz-Hausdorff Theorem in our proof. However, the motivation behind our approach is different compared to the approaches taken in \cite{BS,BP,K,TA} to prove the Bhatia-\v{S}emrl Theorem.\\

The present work is organized in the following manner. In the second section, we investigate the directional orthogonality in a complex Banach space. In the third section, we describe the directional orthogonality in terms of linear functionals and characterize smoothness of an element in the underlying space. In the final section, we study Birkhoff-James orthogonality of linear operators between finite-dimensional complex Banach spaces. We apply the ideas developed in this article, along with the Toeplitz-Hausdorff Theorem, to prove the Bhatia-\v{S}emrl Theorem.

\section{Directional orthogonality in complex Banach spaces}

Let us begin with an easy proposition. The results of this proposition will be used extensively throughout this article. We omit the proof of this proposition as it is trivial. We would like to mention that some of the statements of the following proposition are already mentioned in Proposition 2.1 of \cite{PSMM}.

\begin{prop}\label{basic preliminaries}
Let $ \mathbb{X} $ be a complex Banach space and let $ \gamma\in S^1 $. Then for any $ x, y \in \mathbb{X} $, the following hold true:\\
(i) Either $ y\in (x)_\gamma^{+} $, or $ y\in (x)_\gamma^{-} $.\\
(ii) $(x)_\gamma^{+}=(x)_{-\gamma}^{-}. $ \\
(iii) $ x\perp_\gamma y $ if and only if $ x\perp_{-\gamma} y $ if and only if $ y\in (x)_\gamma^{+} \cap (x)_\gamma^{-} $ .
\end{prop}

Our next observation is also somewhat expected. For a given pair of non-zero elements $ x $ and $ y $, in a complex Banach space $ \mathbb{X} $, it may happen that $ x\perp_\gamma y, $ for some $ \gamma\in S^1 $ but $ x\not\perp_B y $. We furnish the following easy example in support of our statement.

\begin{example}
Let $ \mathbb{X} $ be the two-dimensional complex inner product space with the usual inner product and let $ x = ( 1,0 ) $ and $ y = ( 1, i) $. Then for any $ t\in \mathbb{R} $, $$ \| x + t i y \| = \| ( 1 + t i , - t ) \| = ( 1 + 2 t^2 )^{\frac{1}{2}} \geq \| x \|. $$ In other words, $ x\perp_i y $. However, $ x\not\perp_B y $ as $ \langle x, y \rangle = 1$.  
\end{example}  

We present the first non-trivial result of the article, which is also geometrically illustrating.

\begin{theorem}\label{existence of direction}
Let $ \mathbb{X} $ be a complex Banach space and let $ x , y \in \mathbb{X} $ be non-zero. Then there exists $ \beta\in S^1 $, such that $ x\perp_{\beta} y $. Moreover, given any $\gamma\in S^1$, if $ y = \lambda x ,$ for some $ \lambda \in \mathbb{C}\setminus \{ 0 \} $, then $ x\perp_\gamma y $ if and only if $ \lambda\gamma $ is purely imaginary.
\end{theorem}

\begin{proof}
If $x\perp_B y$, then $\beta_0:=1$ works. If possible, suppose, $ x\not \perp_\beta y $, for any $ \beta \in S^1 $. We now consider two subsets $W_1$ and $W_2$ of $S^1$, defined as follows:
\begin{align}\label{separation}
& W_1 := \{ \beta \in S^1 :~ y \in (x)_\beta ^-\setminus (x)_\beta ^+  \},
& W_2 := \{ \beta \in S^1 :~ y \in (x)_\beta ^+\setminus (x)_\beta ^-  \}.
\end{align}
\noindent It follows from Proposition \ref{basic preliminaries} (i) that $ W_1 \cup W_2 = S^1 $. Also, it is easy to see that $ W_1 $ and $ W_2 $ are non-empty. Next, we show that $W_1$ and $W_2$ are closed subsets of $S^1.$ Assume that $(\beta_n) \subseteq W_1$ with $\beta_n \to \gamma$. Then 
\begin{align*}
\|x+t\beta_n y\| \geq \|x\|,~~\mathrm{for~all}~~ t\leq 0~\mathrm{and}~n\in \mathbb{N}.    
\end{align*}
Hence $\|x+t\gamma y\|=\underset{n}{\lim}\|x+t\beta_n y\|\geq \|x\|$ for every $t\leq 0.$
As a result, $y \in (x)_\gamma^-$.
Since $x\not\perp_\gamma y$, we have $y \notin (x)_\gamma^+$. Hence $\gamma \in W_1$. Thus, $W_1$ is closed. Similarly, we can show that $W_2$ is closed. Also, it follows from the definition of $W_1$ and $W_2$ that $W_1 \cap W_2 = \emptyset.$ Thus, so far we know that the sets $W_1$, $W_2$ are closed, disjoint, and $W_1\cup W_2=S^1.$ Hence both $W_1$, $W_2$ are also open. And from this we can easily get a contradiction. Therefore, there exists $\beta \in S^1,$ such that $x\perp_{\beta}y.$\\

For the next part of the theorem, we assume that $ x\perp_\gamma \lambda x, $ for some $\gamma \in S^1$. Let $\gamma\lambda=a+ib$, for some $a,b\in \mathbb{R}$. Note that
$$\| x + t(a+ib) x \| = \| x \| |1+ta+tib | = \|x\| \left( 1 +2ta + t^2(a^2+b^2) \right)^{\frac{1}{2}},$$
\noindent for all $t\in \mathbb{R}.$ If $a \neq 0$, then choosing $t = -\dfrac{a}{a^2+b^2} $, we have $$ 1 +2t a + t^2(a^2+b^2) = 1 - \dfrac{a^2}{a^2+b^2} < 1. $$ Consequently, 
$$\|x+t\gamma y\|= \left\| x - \dfrac{a}{a^2+b^2}(a+ib) x \right\| < \| x \|,$$
a contradiction with $x\perp_\gamma y$. Therefore, we must have $a=0.$ As a result, $ \lambda\gamma $ is purely imaginary.

Conversely, suppose that $ \lambda\gamma $ is purely imaginary. Then for every $t\in \mathbb{R}$ we have
$$\| x + t\lambda \gamma x \| = \left( 1 + (t|\lambda \gamma |)^2\right)^{\frac{1}{2}} \| x \| \geq \| x \|.$$
Therefore, $ x\perp_\gamma \lambda x $.
 
\end{proof}

A natural question in view of the above result would be regarding the structure of all possible directions of orthogonality for a given pair of vectors. Our next theorem addresses this query. We need the following definition:

\begin{definition}
Let $ \beta \in S^1 $. The \emph{half-circle determined} by $ \beta $ is a subset of $ S^1 $, defined by
\[ U_\beta := \{ \gamma\in S^1 :~ \mathsf{arg}~\gamma \in [ \mathsf{arg}~\beta ,\pi + \mathsf{arg}~\beta ] \} .\]
\end{definition}

Note that $-U_\beta= U_{-\beta}$, for any $\beta\in S^1.$ We now establish a lemma regarding the direction of orthogonality of two vectors.

\begin{lemma}\label{y belongs to negative of all direction}
Let $ \mathbb{X} $ be a complex Banach space and let $ x , y \in \mathbb{X} $. Let $ x\perp_{\beta_0} y, $ for some $ \beta_0\in S^1 $. Let $ \gamma_0\in U_{\beta_0} $ be such that $ y\in (x)_{\gamma_0}^-\setminus (x)_{\gamma_0}^+  \left( y\in (x)_{\gamma_0}^+\setminus (x)_{\gamma_0}^- \right)$. Then $ y\in (x)_{\gamma}^- \left( y\in (x)_{\gamma}^+ \right),$ for all $ \gamma\in U_{\beta_0} $.
\end{lemma}

\begin{proof}
Since $ y\notin (x)_{\gamma_0}^+ $, there is $ t_0 \in (0,1) $ such that $$ \| x + t_0 \gamma_0 y \| < \| x \| .$$ Consider any $t_1\in (0,t_0)$. Put $\lambda_0:=\frac{t_1}{t_0}$; thus $\lambda_0\in(0,1)$. Observe that
$$ x + t_1 \gamma_0 y =(1-\lambda_0)x+\lambda_0( x + t_0 \gamma_0 y).$$
Considering norm on the both sides and applying triangle inequality, we obtain
\begin{align}\label{convexity of norm}
\| x + t_1 \gamma_0 y \| < \| x \|.
\end{align}
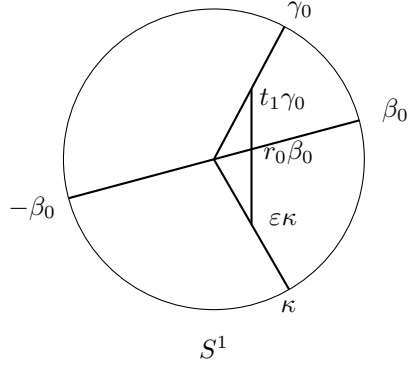
\begin{figure}[h]
\begin{center}
\begin{tikzpicture}
\draw[black] (0,0) circle[radius=2];
\draw[thick] (0:0) -- ++(15:2);
\draw (15:2.5) node{$\beta_0$};
\draw[thick] (0:0) -- ++(195:2);
\draw (195:2.50) node{$-\beta_0$};
\draw[thick] (0:0) -- ++(62:2);
\draw (60: 2.3) node{$\gamma_0$};
\draw[thick] (0:0) -- ++(-60:2);
\draw (-90: 2.5) node{$S^1$};
\draw[thick] (-60:1) -- ++(90:1.8);
\draw (0.92, 0.8) node{$t_1\gamma_0$};
\draw (5:1) node{$r_0\beta_0$};
\draw (0.92, -0.8) node{$\varepsilon \kappa$};
\draw (0.99, -1.95) node{$\kappa$};
\end{tikzpicture}
\end{center}
\caption{Positive and negative part along a half-circle.}
\label{Positive and negative part along a half circle} 
\end{figure}

Let $ \kappa \in U_{-\beta_0}\setminus \{\pm \beta_0\} $. Let $ \varepsilon\in (0,1) $ be arbitrary. We now consider two subsets $\mathsf{A}$ and $\mathsf{B}$ of $\mathbb{C}$, defined as:
$$ \mathsf{A}:=\{ \mu (\varepsilon \kappa) + (1-\mu) (t_1 \gamma_0) :~ \mu \in [0,1] \},\qquad \mathsf{B}:= \{ r \beta_0 :~r\in [-1,1] \}.$$
It follows from the definition (Figure \ref{Positive and negative part along a half circle}) of $\mathsf{A}$ and $\mathsf{B}$ that there exist some $\mu_0\in (0,1]$ and $r_0\in [-1,1]$ such that $$\mathsf{A}\cap \mathsf{B}= \{\mu_0 (\varepsilon \kappa) + (1-\mu_0) (t_1 \gamma_0)\} = \{r_0 \beta_0\}.$$ Now,
\begin{align*}
\| x + r_0 \beta_0 y \| & = \| x + \left(\mu_0 (\varepsilon \kappa) + (1-\mu_0) (t_1 \gamma_0)\right)y \|\\
& = \| \mu_0 x + (1-\mu_0) x + \left(\mu_0 (\varepsilon \kappa) + (1-\mu_0) (t_1 \gamma_0)\right)y \|\\
& \leq \| \mu_0 x + \mu_0 (\varepsilon \kappa) y \| + \| (1-\mu_0) x + (1-\mu_0) (t_1 \gamma_0)y \|\\
& =  \mu_0  \| x +  (\varepsilon \kappa) y \| +   (1- \mu_0) \| x + t_1 \gamma_0 y \|\\
& < \mu_0 \| x + (\varepsilon \kappa) y \| + (1-\mu_0)\| x \|;~~\mathrm{(using~\ref{convexity of norm})}.
\end{align*}
Therefore,
\begin{align}\label{Obtaining inequality}
\| x + r_0 \beta_0 y \|-(1-\mu_0)\| x \| < \mu_0 \| x + (\varepsilon \kappa) y \|.
\end{align}
\noindent By our assumption $x\perp_{\beta_0} y $. Therefore, (\ref{Obtaining inequality}) produces 
$$ \| x \| < \| x + \varepsilon \kappa y \|.$$ 
This means that $ \| x + t \kappa y \| \geq \| x \| ,$ for all $ t \geq 0 $. Therefore, by $(ii)$ of Proposition \ref{basic preliminaries}, $ y\in (x)_\kappa^+=(x)_{-\kappa}^- $ for all $ \kappa \in -U_{\beta_0}.$ Thus, $y \in (x)_{\gamma}^-$ for all $ \gamma\in U_{\beta_0} $.
 
\end{proof}

Now, we prove the promised theorem.
\begin{theorem}\label{Directions of orthogonality}
Let $ \mathbb{X} $ be a complex Banach space and let $ x , y \in \mathbb{X} $ be non-zero with $x\not\perp_B y.$ Then $ \{ \gamma \in S^1 :~ x\perp_\gamma y \} $ is the union of two diametrically opposite closed arcs of the unit circle $ S^1 $.
\end{theorem}

\begin{proof}
Let $\mathsf{S}$ be a subset of $S^1$, defined by $$\mathsf{S}:=\{ \gamma \in S^1 :~ x\perp_\gamma y \} .$$
It is sufficient to show that $\mathsf{S} = \mathsf{E}\cup (-\mathsf{E}),$ where $ \mathsf{E} $ is a closed and connected subset of $ S^1 $. It follows from Theorem \ref{existence of direction} that $ x\perp_{\beta_0} y $ (hence $x\perp_{-\beta_0}y$), for some $ \beta_0 \in S^1 $. If $\mathsf{S} =\{\beta_0,-\beta_0\}$, then we are done. So we assume that $\{\beta_0,-\beta_0\}\subsetneq \mathsf{S} .$ Without loss of generality, let $  \beta_0 := 1 $. We now complete the proof in the following three steps:\\

\noindent Step I: Since $ x \not\perp_B y $, there exists $ \gamma_0 \in U_{\beta_0} $ such that $ x \not\perp_{\gamma_0} y $. Without loss of generality, let $ y\in (x)_{\gamma_0}^-\setminus (x)_{\gamma_0}^+ $. Next, we consider two subsets $\mathsf{C}$ and $\mathsf{D}$ of $S^1$, defined as: 
\begin{align*}
&\mathsf{C} := \{ \beta\in S^1 :~ x\perp_\beta y,~ \mathsf{arg}~\beta\in [0,~ \mathsf{arg}~\gamma_0]\},\\
&\mathsf{D} := \{ \beta\in S^1 :~ x\perp_\beta y,~ \mathsf{arg}~\beta\in [ \mathsf{arg}~\gamma_0, \pi]\}.
\end{align*}

We show below that $\mathsf{C}$ and $\mathsf{D}$ are non-empty and closed (hence compact) subsets of $S^1.$ It is trivial to see that $\mathsf{C}$ and $\mathsf{D}$ are non-empty, as $\beta_0\in \mathsf{C}$ and $-\beta_0\in \mathsf{D}$. Let $(\beta_n)$ be a sequence in $\mathsf{C}$ with $\beta_n\to \alpha.$ Then 
\begin{align*}
\|x+t\alpha y\|=\underset{n}{\lim}\|x+t\beta_n y\| \geq \|x\|,~~\mathrm{for~all}~~ t\in \mathbb{R}.
\end{align*}

\begin{figure}[h]
\begin{center}
\begin{tikzpicture}

\draw[black][very thick] (0,0) (0:2cm) arc (0:56:2cm);
\draw[black][very thick] (0,0) (310:2cm) arc (310:360:2cm);
\draw[black][very thick] (0,0) (180:2cm) arc (180:236:2cm);
\draw[black][very thick] (0,0) (490:2cm) arc (490:540:2cm);
\draw[black] (0,0) (56:2cm) arc (56:130:2cm);
\draw[black] (0,0) (236:2cm) arc (236:310:2cm);
\draw[thick] (0:0) -- ++(0:2);
\draw (0: 2.3) node{$\beta_0$};

\draw[thick] (0:0) -- ++(14:2);
\draw (14: 2.3) node{$\beta_1$};

\draw[thick] (0:0) -- ++(28:2);
\draw (28: 2.3) node{$\beta_3$};

\draw[thick] (0:0) -- ++(42:2);
\draw (42: 2.3) node{$\beta_2$};

\draw[thick] (0:0) -- ++(56:2);
\draw (56: 2.3) node{$\alpha_0$};

\draw (-90: 3) node{$S^1$};

\draw[thick] (0:0) -- ++(70:2);
\draw (70: 2.3) node{$\gamma_0$};

\draw[thick] (0:0) -- ++(250:2);
\draw (250: 2.3) node{$-\gamma_0$};

\draw[thick] (0:0) -- ++(130:2);
\draw (130: 2.3) node{$\kappa_0$};

\draw[thick] (0:0) -- ++(236:2);

\draw[thick] (0:0) -- ++(310:2);

\draw[thick] (0:0) -- ++(180:2);
\draw (-2.50,0.0) node{$-\beta_0$};

\draw (30:3) node{${\mathsf{C}}$};
\draw (160:3) node{${\mathsf{D}}$};
\draw (210:3) node{$-{\mathsf{C}}$};
\draw (340:3) node{$-{\mathsf{D}}$};
\end{tikzpicture}
\end{center}
\caption{Direction of orthogonality.} 
\label{Direction of orthogonality}
\end{figure}
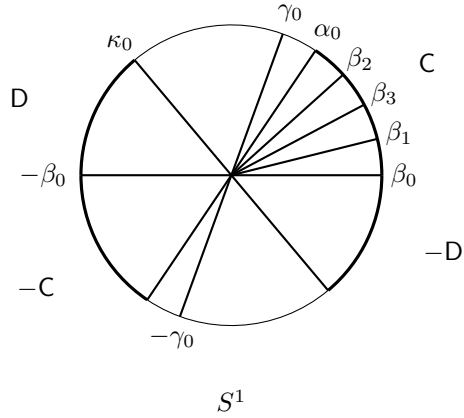
\noindent Therefore, $x\perp_\alpha y$. Observe that $\mathsf{arg}~{\gamma_0}\in (0, \pi).$ For every $\theta\in \mathbb{R}$, define $e^{i\theta}:=(\cos\theta, \sin\theta)\in \mathbb{R}^2.$ Since the argument function $\mu\mapsto \mathsf{arg}~\mu$ is continuous on $\{e^{i\theta}: 0\leq \theta \leq \mathsf{arg}~\gamma_0\}$, we have $\mathsf{arg}~\beta_n \to \mathsf{arg}~{\alpha}$. As a result, $0 \leq \mathsf{arg}~\alpha \leq \mathsf{arg}~\gamma_0$. Thus, $\alpha\in \mathsf{C}$ and $ \mathsf{C} $ is closed, as expected. A similar argument shows that $\mathsf{D}$ is also closed.\\

\noindent Step II: We show here that $\mathsf{C}$ and $\mathsf{D}$ are connected subsets of $S^1.$ If $\mathsf{C}:=\{\beta_0\}$, then we have nothing more to show. Assume that there are $\beta_1, \beta_2\in \mathsf{C}$ with $\mathsf{arg}~\beta_2 > \mathsf{arg}~\beta_1.$ Let $t_0\in (0,1)$ be arbitrary and let $$\beta_3 := \dfrac{t_0\beta_2+(1-t_0)\beta_1}{|t_0\beta_2+(1-t_0)\beta_1|}.$$
Observe that $\gamma_0, \beta_3 \in U_{\beta_1}$ (Figure \ref{Direction of orthogonality}) and $ y\in (x)_{\gamma_0}^-\setminus (x)_{\gamma_0}^+ $. Thus, it follows from Lemma \ref{y belongs to negative of all direction} that $y\in (x)_{\beta_3}^-.$ Also, $-\gamma_0, \beta_3 \in U_{-\beta_2}$ (Figure \ref{Direction of orthogonality}) and $ y\in (x)_{-\gamma_0}^+\setminus (x)_{-\gamma_0}^- $. Therefore, it follows from Lemma \ref{y belongs to negative of all direction} that $y\in (x)_{\beta_3}^+.$ Now, applying Proposition \ref{basic preliminaries} (iii), we have that $x\perp_{\beta_3} y.$ Since $\mathsf{arg}~{\beta_3} < \mathsf{arg}~{\beta_2} < \mathsf{arg}~{\gamma_0}$, we have $\beta_3\in \mathsf{C}$. Therefore, $\mathsf{C}$ is connected. A similar argument shows that $\mathsf{D}$ is connected.

In other words, we can find $\alpha_0,\kappa_0\in S^1$ such that $\mathsf{C}$ is the closed arc between $\beta_0$ and $\alpha_0$ and $\mathsf{D}$ is the closed arc between $\kappa_0$ and $-\beta_0$ (Figure \ref{Direction of orthogonality}).\\

\noindent Step III: Finally, observe that 
$$\mathsf{S} = \mathsf{C}\cup (-\mathsf{C})\cup \mathsf{D} \cup (-\mathsf{D}).$$
Clearly, $\mathsf{C}\cap (-\mathsf{D})=\{\beta_0\}$. As a result, $\mathsf{C}\cup (-\mathsf{D})$ is a closed and connected subset of $S^1$. Now, put $\mathsf{E}:=\mathsf{C}\cup (-\mathsf{D}).$ Therefore, $\mathsf{S}= \mathsf{E}\cup (-\mathsf{E})$, where $\mathsf{E}$ is a closed and connected subset of $S^1$.
 
\end{proof}

\section{Directional orthogonality and smoothness}

This section plans to characterize the smoothness of a non-zero element $x$ in a complex Banach space $\mathbb{X},$ using the directional orthogonality. To fulfill the desired goal, we need to characterize the directional orthogonality in terms of linear functionals on $\mathbb{X}.$

For $x\in \mathbb{X}$, put
\begin{align*}
\mathsf{J}(x) := \{x^*\in S_{\mathbb{X}^*} :~ |x^*(x)| = \|x\| \}.
\end{align*}
Note that $\mathsf{J}(x)$ is different from $\mathbb{J}(x).$
For $\mu\in S^1$ define
\begin{align}\label{norming functionals for a direction}
\mathsf{J}_\mu(x) := \{x^*\in S_{\mathbb{X}^*} :~ x^*(x) = \mu\|x\| \}.
\end{align}
It is well-known from the Banach-Alaoglu Theorem that the closed unit ball $B_{\mathbb{X}^*}$ of $\mathbb{X}^*$ is weak$^*$ compact. Also, it is easy to see that  $\mathsf{J}_\mu(x)$ is non-empty, convex and weak$^*$ compact.

The following theorem provides a necessary and sufficient condition for orthogonality of two vectors $x,y$ in a particular direction. 

\begin{theorem}\label{functional Characterization}
Let $\mathbb{X}$ be a complex Banach space. Let $x,y\in \mathbb{X}$ be non-zero vectors and let $\mu\in S^1.$ Then $x\perp_\mu y$ if and only if there exists $u^*\in \mathsf{J}_\mu(x)$ such that $\mathsf{Re}~u^*(y)=0$.
\end{theorem}

\begin{proof}
Note that the directional orthogonality is real homogeneous. Therefore, without loss of generality, we may and do assume that $\|x\|=\|y\| = 1$. We first prove the necessity part. Clearly, $x\neq r\mu y$, for any $r\in \mathbb{R}$. If possible, suppose $\mathsf{Re}~x^*(y) \neq 0$, for every $x^*\in \mathsf{J}_\mu(x)$. We will proceed in three steps:\\ 

Step I: Define the set
\begin{align*}
W := \{ \mathsf{Re}~p^*(y) :~ p^*\in \mathsf{J}_\mu(x)\}\subseteq \mathbb{R}.
\end{align*}
The set $\mathsf{J}_\mu(x)$ being clearly convex and weak$^*$ compact (hence connected), $W$ must be a compact interval in $\mathbb{R}$. Clearly, $0 \notin W.$ Therefore, every member of $W$ is of the same sign. Without loss of generality, we assume that every member of $W$ is positive. Due to the compactness of $W$, we can find some $\lambda \in (0,1)$ such that $w>\lambda$, for all $w\in W$. In other words, 
\begin{align}\label{condition on Jmu(x)}
\mathsf{Re}~p^*(y) > \lambda,~\mathrm{for~all~} p^*\in \mathsf{J}_\mu(x).
\end{align}

Step II: Define
\begin{align*}
G := \left\{x^*\in S_{\mathbb{X}^*} :~ \mathsf{Re}~\mu\overline{x^*(x)}x^*(y) \leq \dfrac{\lambda}{2}\right\}.
\end{align*}
Clearly, $|x^*(x)| \leq 1$ for every $x^*\in G$. We will show that 
$$\sup\{|x^*(x)| :~ x^*\in G\} < 1-2\varepsilon,$$
for some $\varepsilon \in (0,\frac{1}{2})$. If possible, suppose $\sup\{|x^*(x)| :~ x^*\in G\}=1.$ We consider two cases.\\

\noindent Case I: Suppose there exists $x_0^*\in G$ such that $|x_0^*(x)|=1$. Then $x_0^*\in \mathsf{J}(x)$. Observe that $\mu\overline{x_0^*(x)}x_0^*\in \mathsf{J}_\mu(x)$. Therefore, $\mathsf{Re}~\mu\overline{x_0^*(x)}x_0^*(y)\in W$. In particular, it follows from (\ref{condition on Jmu(x)}) that $\mathsf{Re}~\mu\overline{x_0^*(x)}x_0^*(y) > \lambda $, which is a contradiction.\\

\noindent Case II: Suppose there exists a sequence $(x_n^*)$ in $G$ such that $|x_n^*(x)|\to 1$. Find a subnet $(x_{n_\tau}^*)_{\tau\in \mathscr{T}}$ of the sequence $(x^*_n)$ which converges to
an $x^*_0\in \mathbb{X}^*$, say. Then $|x_0^*(x)|=1$ and therefore, $x_0^*\in S_{\mathbb{X}^*}$. Consequently, $x_0^*\in \mathsf{J}(x)$ and $\mu\overline{x_0^*(x)}x_0^*\in \mathsf{J}_\mu(x)$. Again, it follows from (\ref{condition on Jmu(x)}) that $\mathsf{Re}~\mu\overline{x_0^*(x)}x_0^*(y) > \lambda $. However, this is a contradiction, since $\mathsf{Re}~\mu\overline{x_n^*(x)}x_n^*(y) \leq \dfrac{\lambda}{2}$, for every $n\in \mathbb{N}$. Therefore,  $\sup\{|x^*(x)| :~ x^*\in G\} < 1-2\varepsilon$, for some $\varepsilon \in (0,\frac{1}{2})$, as desired.\\

Step III: In this step, we find a real number $\varepsilon_0>0$ such that $\|x-\varepsilon_0\mu y\| < 1,$ which will contradict to $x\perp_\mu y$. Choose $0< \varepsilon_0 < \min\{\frac{\lambda}{2}, \varepsilon\}$. Now, for any $x^*\in G$,
\begin{align*}
|x^*(x-\varepsilon_0\mu y)|\leq |x^*(x)|+\varepsilon_0|\mu||x^*(y)| < 1-2\varepsilon + \varepsilon_0 < 1-\varepsilon.
\end{align*} 
If $x^*\in S_{\mathbb{X}^*}\setminus G$, then $\mathsf{Re}~\mu\overline{x^*(x)}x^*(y) > \dfrac{\lambda}{2}$, and so
\begin{align*}
|x^*(x-\varepsilon_0\mu y)|^2 & = x^*(x-\varepsilon_0\mu y)\overline{x^*(x-\varepsilon_0\mu y)}\\
& = |x^*(x)|^2+\varepsilon_0^2|x^*(y)|^2 - 2\varepsilon_0\mathsf{Re}~\mu\overline{x^*(x)}x^*(y)\\
& < 1 + \varepsilon_0^2 - \lambda\varepsilon_0.
\end{align*}
As $\varepsilon_0 < \lambda $, choose, $0<\delta_0 < \varepsilon$ such that $(1-\delta_0)^2> 1 + \varepsilon_0^2 - \lambda\varepsilon_0$. As a consequence,
\begin{align*}
|x^*(x-\varepsilon_0\mu y)| < 1-\delta_0, ~~\mathrm{for~all}~x^*\in S_{\mathbb{X}^*}.
\end{align*}
Thus, 
$$\|x-\varepsilon_0\mu y\|=\sup\{|x^*(x-\varepsilon_0\mu y)|:~ x^*\in S_{\mathbb{X}^*}\}\leq 1-\delta_0 < 1=\|x\|.$$ Therefore, $x$ is not $\mu$ orthogonal to $y$, a contradiction.

To prove the sufficiency part, we proceed as follows:
\begin{align*}
\|x+t\mu y\|^2 & \geq |u^*(x+t\mu y)|^2\\
& = |\mu +t \mu u^*(y)|^2\\
& = |1 +t u^*(y)|^2\\
& = (1 +t u^*(y))(1+t\overline{u^*(y)})\\
& = 1+2t\mathsf{Re}~u^*(y)+t^2|u^*(y)|^2\\
& \geq 1,
\end{align*}
for all $t\in \mathbb{R}$.
 
\end{proof}

After formulating the directional orthogonality of two vectors in terms of linear functionals, we now turn our attention towards characterizing the smoothness of a non-zero element $x$ in a complex Banach space $\mathbb{X}.$ We introduce the following definition to serve our purpose. 

\begin{definition}(\emph{Orthogonality pair})\label{Orthogonality pair}
Let $\mathbb{X}$ be a complex Banach space and let $(x,y)\in \mathbb{X}\times \mathbb{X}$. An ordered pair $(\mu,x^*)\in S^1\times S_{\mathbb{X}^*}$ is said to be an \emph{orthogonality pair} for $(x,y)$ if 
\begin{align*}
x^*(x) = \mu\|x\|~~\mathrm{and}~~\mathsf{Re}~x^*(y)=0.
\end{align*}
We denote the collection of all orthogonality pairs for $(x,y)$ by $\mathcal{O}(x,y)$.
\end{definition}

\begin{remark}
Let $(x,y) \in \mathbb{X}\times\mathbb{X}.$ Then  Theorem \ref{existence of direction} ensures the existence of a uni-modular constant $\mu$, such that $x\perp_\mu y.$ Now, applying Theorem \ref{functional Characterization} we can find a unit vector $x^*\in S_{\mathbb{X}^*}$ such that $$x^*(x)=\mu\|x\|~\mathrm{and}~\mathsf{Re}~x^*(y)=0.$$ This shows that the ordered pair $(\mu,x^*) \in S^1\times S_{\mathbb{X}^*}$ is a member of $\mathcal{O}(x,y).$ It follows that $\mathcal{O}(x,y)$ is non-empty, for any $(x,y) \in \mathbb{X}\times\mathbb{X}.$ Note that the cardinality of $\mathcal{O}(x,y)$ is always greater than or equals to $2,$ as $(\mu,x^*) \in \mathcal{O}(x,y)$ implies that $(-\mu,-x^*) \in \mathcal{O}(x,y).$
\end{remark}

We now present the promised characterization:

\begin{theorem}\label{smoothness characterization}
Let $\mathbb{X}$ be a complex Banach space and let $x\in \mathbb{X}$ be non-zero. Then $x$ is smooth if and only if for any $y\in \mathbb{X}$ with $x\not\perp_B y$, $\mathcal{O}(x,y)=\{ (\mu, u^*),(-\mu, -u^*)\}$, for some $(\mu,u^*)\in S^1\times S_{\mathbb{X}^*}$.
\end{theorem}

\begin{proof}
Without loss of generality, let $\|x\|=1$. We first prove the necessity part. Consider any $y\in \mathbb{X}$ with $x\not\perp_B y$. If possible, suppose $(\alpha, x^*),(\beta, y^*)\in \mathcal{O}(x,y)$ with $(\beta, y^*) \neq (\alpha, x^*), (-\alpha, -x^*)$, for some $(\alpha,x^*),( \beta,y^*)\in S^1\times S_{\mathbb{X}^*} $. Let $\mathbb{J}(x):=\{v^*\},$ for some $v^*\in S_{\mathbb{X}^*}.$ It is easy to see that $$x^*=\alpha v^*~\mathrm{and}~y^*=\beta v^*.$$ Indeed, if $x^*\neq \alpha v^*$ $(y^*\neq \beta v^*)$, then $\overline{\alpha}x^*$ $(\overline{\beta}y^*)$ is a member of $\mathbb{J}(x)$ with $\overline{\alpha}x^*\neq v^*$ $(\overline{\beta}y^*\neq v^*)$. This contradicts the fact that $x$ is smooth. It follows from our hypothesis that $$\mathsf{Re}~\alpha v^*(y)=\mathsf{Re}~\beta v^*(y)=0.$$ Since $S^1$ is a group under multiplication, there exists $\sigma\in S^1$ such that $\sigma \alpha = \beta.$ We claim that $\sigma \neq \pm 1$. If $\sigma = \pm 1$, then $\beta = \pm \alpha$. If $\beta = \alpha$, then
\begin{align*}
(\beta, y^*) = (\alpha, y^*) = (\alpha, \alpha v^*) = (\alpha, x^*).
\end{align*}
Also, if $\beta = -\alpha$, then
\begin{align*}
(\beta, y^*) = (-\alpha, y^*) = (-\alpha, -\alpha v^*) = (-\alpha, -x^*).
\end{align*}
This is a contradiction to the fact that $(\beta, y^*) \neq (\alpha, x^*), (-\alpha, -x^*)$. Thus, $\sigma \neq \pm 1$, as expected. Now, it follows from our assumption that
\begin{align*}
\mathsf{Re}~x^*(y)=0~~\mathrm{and}~~\mathsf{Re}~\sigma x^*(y)= \mathsf{Re}~\sigma\alpha v^*(y) = \mathsf{Re}~\beta v^*(y) = \mathsf{Re}~y^*(y)=0.
\end{align*}
Let $\sigma =a+ib$, for some real numbers $a,b$. Observe that $b\neq 0$, as $\sigma \neq \pm 1.$ Therefore, 
\begin{align*}
\mathsf{Im}~x^*(y)=\mathsf{Re}~\dfrac{1}{b}(a-\sigma)x^*(y)= \dfrac{a}{b}\mathsf{Re}~ x^*(y)-\dfrac{1}{b} \mathsf{Re}~\sigma x^*(y) = 0.
\end{align*}
As a result, $\alpha v^*(y)=0$. Consequently, $x\perp_By$, which is a contradiction.\\

We now prove the sufficiency part. If possible, suppose $x^*,y^*$ are two distinct members of $\mathbb{J}(x)$. Then we can find $\alpha, \beta\in S^1$ such that $\mathsf{Re}~\alpha x^*(y)= \mathsf{Re}~\beta y^*(y)=0.$ Consequently, $(\alpha, \alpha x^*)$ and $(\beta, \beta y^*)$ are members of $\mathcal{O}(x,y).$ It follows from the hypothesis of the theorem that either $(\beta, \beta y^*) = (\alpha, \alpha x^*)$, or, $(\beta, \beta y^*)=(-\alpha, -\alpha x^*)$. However, in either case, we have $x^*=y^*.$ Therefore, the assumption that $x^*,y^*$ are distinct members of $\mathbb{J}(x)$ is not tenable. Consequently, $\mathbb{J}(x)$ is singleton and $x$ is smooth.
 
\end{proof}

As an obvious application of Theorem \ref{smoothness characterization}, we now characterize smooth Banach spaces among all complex Banach spaces.

\begin{cor}\label{smooth space}
Let $\mathbb{X}$ be a complex Banach space. Then $\mathbb{X}$ is smooth if and only if for any ordered pair $(x,y)\in \mathbb{X} \times \mathbb{X},$ with $x\not\perp_B y$, $\mathcal{O}(x,y)=\{ (\mu, u^*),(-\mu, -u^*)\}$, for some $(\mu, u^*)\in S^1\times S_{\mathbb{X}^*}$.  
\end{cor}

Let $\mathbb{X}$ be a  complex Banach space. Let $x,y \in \mathbb{X}$ be non-zero with $x$ being smooth. Our next corollary is about the directions along which $x$ is orthogonal to $y.$ We omit the proof, as it follows directly from the necessary part of Theorem \ref{smoothness characterization}. 

\begin{cor}\label{direction and smoothness}
Let $\mathbb{X}$ be a complex Banach space. Let $x,y \in \mathbb{X}$ be non-zero with $x$ being smooth. Let $\mathsf{S}:=\{\gamma \in S^1 :~ x\perp_\gamma y \}.$ Then either the cardinality of $\mathsf{S}$ is $2,$ or, $\mathsf{S}=S^1.$ 
\end{cor}

Since every Hilbert space is smooth, we have the following:
 
\begin{cor}\label{direction and Hilbert space}
Let $\mathbb{H}$ be a complex Hilbert space and let $x,y \in \mathbb{H}$ be non-zero. Let $\mathsf{S}:=\{\gamma \in S^1 :~ x\perp_\gamma y \}.$ Then either the cardinality of $\mathsf{S}$ is $2,$ or, $\mathsf{S}=S^1.$  
\end{cor} 

The above two corollaries highlight the distinction between a smooth point and a non-smooth point in a Banach space. As a whole, the above results reflect the structural differences between complex Hilbert spaces and complex Banach spaces. It is needless to mention that the concept of the directional orthogonality facilitates capturing the aforesaid geometric dissimilarities.

\section{Directional orthogonality and Hilbert spaces}

In a complex Banach space, in general it is difficult to explicitly identify the directions along which two given vectors are orthogonal. However, in case of a complex Hilbert space, we have a straightforward description of the latter, which we present in the following theorem.\\

\begin{theorem}\label{Characterization of orthogonality in Hilbert space in a particular direction}
Let $ \mathbb{H} $ be a complex Hilbert space and let $ x, y $ be two non-zero elements in $ \mathbb{H} $. Then for some $ \gamma \in S^1 $, $ x\perp_\gamma y $ if and only if $ \mathsf{Re}~{\gamma} \langle y, x \rangle = 0 $. 
\end{theorem}

\begin{proof}
Without loss of generality assume that $\|x\|=1$. The verification of the sufficiency is straightforward. Let us prove the necessity. It follows from Theorem \ref{functional Characterization} that there exists $u^*\in S_{\mathbb{H}^*}$, satisfying 
$$u^*(x)=\gamma ~\mathrm{and}~\mathsf{Re}~u^*(y)=0.$$ 
Thus, $\overline{\gamma}u^*\in \mathbb{J}(x).$ And, as the underlying space is Hilbert, necessarily 
$\overline{\gamma}u^*(z)=\langle z,x \rangle$, for all $z \in \mathbb{H}$. Hence $u^*(y) = \gamma \langle y,x\rangle $, and so $\mathsf{Re}~\gamma \langle y,x\rangle = \mathsf{Re}~u^*(y)= 0.$
 
\end{proof}

\begin{cor}\label{relation of orthogonality between x and y in a Hilbert space in a particular direction}
Let $ \mathbb{H} $ be a complex Hilbert space and let $ x, y $ be two non-zero elements in $ \mathbb{H} $. Let $ \gamma \in S^1 $. Then $ x\perp_\gamma y $ if and only if $ y \perp_{\overline{\gamma}} x $.
\end{cor}

\subsection{Directional orthogonality and Bhatia-\v{S}emrl Theorem}
In Theorem $ 2.6 $ of \cite{PSMM}, a Bhatia-\v{S}emrl type result has been given in the context of complex Banach spaces. In particular, Proposition $ 2.1 $ of \cite{TA} follows as a simple corollary to this result. In this article, we illustrate that Theorem $ 2.6 $ of \cite{PSMM} should be viewed as a generalization of the Bhatia-\v{S}emrl Theorem to the setting of complex Banach spaces. However, this requires some effort. For the sake of completeness, let us state the concerned theorem from \cite{PSMM}.

\begin{theorem}\label{Relation between direction and M_T}
Let $ \mathbb{X},~\mathbb{Y} $ be finite-dimensional complex Banach spaces and let $ T,A \in \mathbb{L}(\mathbb{X},\mathbb{Y}) $ be  linear operators such that $ M_T $ is connected. Then $ T \perp_B A $ if and only if for each $ \gamma\in S^1 $, there exists $ x\in M_T $ such that $ Tx \perp_\gamma Ax $.
\end{theorem}

As mentioned earlier, we make note of the following corollary which follows from the above theorem. This gives an alternative proof of Proposition $ 2.1 $ of \cite{TA}.

\begin{cor}\label{Generalization of Turnsek's result}
Let $ \mathbb{H} $ be a finite-dimensional complex Hilbert space and let $ T,A \in \mathbb{L}(\mathbb{H}) $. Then $ T\perp_B A $ if and only if for each direction $ \gamma \in S^1 $, there exists $ x\in M_T $ such that $ \mathsf{Re}~{\gamma}\langle Ax, Tx \rangle = 0 $. 
\end{cor}

\begin{proof}
From Theorem 2.2 of \cite{SP}, it follows that $ M_T $ is connected. Therefore, the proof of the corollary follows from Theorem \ref{Characterization of orthogonality in Hilbert space in a particular direction} and Theorem \ref{Relation between direction and M_T}.
 
\end{proof}

As the most important application of our study, we would like to obtain a completely new proof of the Bhatia-\v{S}emrl Theorem, which is strikingly simple and geometrically motivated. Moreover, to our surprise, we observe that it is possible to generalize the celebrated Toeplitz-Hausdorff Theorem and to apply it to serve our purpose. The proof of the following theorem, which generalizes the Toeplitz-Hausdorff Theorem, is strongly inspired by \cite{GK}.
Let $\mathbb{X}$ be a Banach space. Recall that a linear operator $\mathcal{I}\in \mathbb{L}(\mathbb{X})$ is called isometry if $\|\mathcal{I}(x)\|=\|x\|$, for all $x\in \mathbb{X}.$

\begin{theorem}\label{Generalization of Toeplitz-Hausdorff Theorem}
Let $ \mathbb{H} $ be a complex Hilbert space and let $ T,A\in \mathbb{L}(\mathbb{H}) $. Let $\mathbb{H}_0$ be a subspace of $\mathbb{H}$ such that $T$ is a scalar multiple of an isometry on $\mathbb{H}_0.$ Then
$$W_A(T) := \left\{ \langle Ax,Tx \rangle :~ x\in S_{\mathbb{H}_0} \right\}$$
is a convex subset of $\mathbb{C}.$
\end{theorem}

\begin{proof}
 Let $\|Tx\| = c\|x\|, $ for all $x\in \mathbb{H}_0$, and for some $c\geq0.$  If $c=0,$ then $ W_A(T) = \{ 0 \} $, and we have nothing to prove. Let $\mu_1, \mu_2$ be two distinct members of $W_A(T).$ Find $x_1,x_2\in S_{\mathbb{H}_0}$ such that $ \langle Ax_1, Tx_1 \rangle = \mu_1 ~\mathrm{and}~ \langle Ax_2, Tx_2 \rangle = \mu_2.$ We have to show that the linear segment $[\mu_1, \mu_2]$ lies in $W_A(T).$ Observe that for all $y\in S_{\mathbb{H}_0}$, 
\begin{align*}
\left \langle \left (\gamma A+\dfrac{\sigma}{c^2}T \right )y, Ty\right \rangle = \gamma \langle Ay, Ty \rangle +\sigma,~ \mathrm{for~scalars}~ \gamma , \sigma.    
\end{align*}
Let $ \gamma_0 := \dfrac{1}{\mu_1-\mu_2} $, $ \sigma_0 := \dfrac{-\mu_2}{\mu_1-\mu_2} $ and let $ P := {\gamma_0 A + \dfrac{\sigma_0}{c^2}}T $. Then we have that
\begin{align}\label{translation}
W_P(T) = \gamma_0 W_A(T) + \sigma_0~;\quad \langle Px_1, Tx_1 \rangle = 1,~\langle Px_2, Tx_2 \rangle = 0.
\end{align}
It follows from (\ref{translation}) that
$$ [\mu_1, \mu_2]  \subseteq W_A(T)~\mathrm{if~and~only~if}~ [0,1] \subseteq W_P(T).$$

Consider any $\lambda \in [0,1].$ We claim that  there exists $x_0:=bx_1+ax_2$ with suitable real $a,b$ such that $\|x_0\|=1~\mathrm{and}~\langle Px_0,Tx_0\rangle= \lambda;$ thus $\lambda\in W_P(T). $ It is enough to show that the following system of equations
\begin{align}\label{ellipse} 
a^2 + b^2 + 2ab ~\mathsf{Re}~ \langle x_1,x_2 \rangle & = 1,
\end{align}
\begin{align}\label{hyperbola}
b^2 + ab \left\{\left\langle Px_1, Tx_2 \rangle + \langle Px_2,Tx_1\right\rangle \right\} & = \lambda,
\end{align}
\vspace{0.5 cm}
\begin{figure}[h]
\begin{center}
\begin{tikzpicture}
\draw[thick][->] (-3.75,0)--(3.75,0) node[below left]{{a}};
\draw[thick][->] (0,-3.75)--(0,3.75) node[above]{{b}};
\draw[rotate=56,black] circle(2cm and 1cm);
\pgfmathsetmacro{\e}{1.44022}   
    \pgfmathsetmacro{\a}{1}
    \pgfmathsetmacro{\b}{(\a*sqrt((\e)^2-1)} 
    \draw plot[domain=-1.85:1.85] ({\a*sinh(\x)},{\b*cosh(\x)});
    \draw plot[domain=-1.85:1.85] ({\a*sinh(\x)},{-\b*cosh(\x)});

\end{tikzpicture}
\end{center}
\caption{Intersection of ellipse and hyperbola.} 
\label{Intersection of ellipse and hyperbola}
\end{figure}
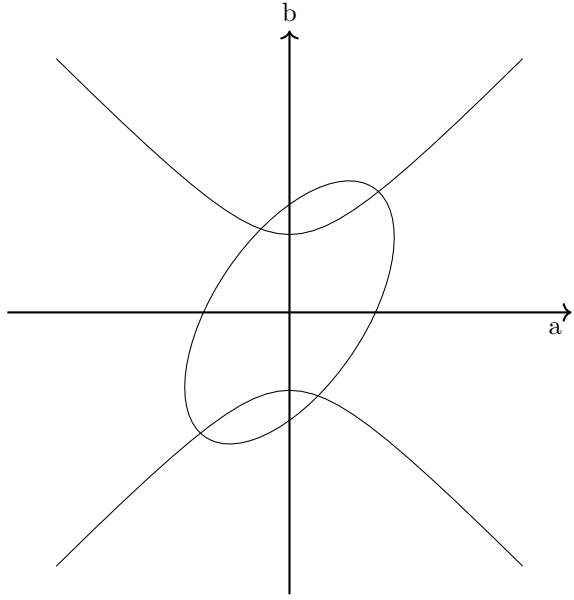
\noindent with unknown $a,b \in \mathbb{R},$ possesses a real solution. We observe that $ | \mathsf{Re}~\langle x_1,x_2 \rangle | < 1 $. Indeed, $ | \mathsf{Re}~\langle x_1, x_2 \rangle | = 1 $ together with the Schwarz's inequality produces $ \langle x_1, x_2 \rangle = \pm 1 $. If $ \langle x_1, x_2 \rangle = 1 $, then we obtain $ \langle x_1-x_2, x_1-x_2 \rangle = 0 $. As a result, $\|x_1-x_2\|=0$ and $x_1=x_2.$ However, this is a contradiction, since $ \langle Px_1, Tx_1\rangle \neq \langle Px_2, Tx_2 \rangle .$ Similarly, considering $ \langle x_1, x_2 \rangle = -1 $, we arrive at the same contradiction. Therefore, $ | \mathsf{Re}~\langle x_1,x_2 \rangle | < 1 $, as expected.\\ Let $$ N_{x_1}:= \left\{\left\langle Px_1, Tx_2 \rangle + \langle Px_2,Tx_1\right\rangle \right\}.$$ If $ N_{x_1} $ is real, then equations (\ref{ellipse}) and (\ref{hyperbola}) represent an ellipse (intercepts $ \pm 1 $ with $a$-axis and $b$-axis) and a hyperbola (intercepts $ \pm \sqrt{\lambda} $ with $b$-axis), respectively (Figure \ref{Intersection of ellipse and hyperbola}).

Observe that $ N_{x_1} $ can always be chosen real by choosing an appropriate scalar multiple of $ x_1 $. Choose $ x_1' = \kappa x_1 $, where $ \kappa = m+in $ with $ m^2+n^2 = 1 $. Note that 
\begin{align*}
\mathsf{Im}~N_{x_1'} & = \mathsf{Im}~\{(m+in) \left\langle Px_1, Tx_2 \rangle + (m-in)\langle Px_2,Tx_1\right\rangle \}\\ 
&= m~\mathsf{Im}~\langle Px_1,Tx_2\rangle +n~\mathsf{Re}~\langle Px_1,Tx_2\rangle+m~\mathsf{Im}~\langle Px_2,Tx_1\rangle -n~\mathsf{Re}~\langle Px_2,Tx_1\rangle\\
&=  m~\mathsf{Im}~N_{x_1} + n~\mathsf{Re}~\{\langle Px_1,Tx_2\rangle - \langle Px_2,Tx_1\rangle \}
\end{align*}
It is not difficult to see that the system of equations
\begin{align}\label{circle}
& m^2 + n^2  = 1,\\ 
&\mathsf{Im}~N_{x_1'} = m~\mathsf{Im}~N_{x_1} + n~\mathsf{Re}~\{\langle Px_1,Tx_2\rangle - \langle Px_2,Tx_1\rangle \} = 0, 
\end{align}
possesses two solutions. As a result, the system of equations (\ref{ellipse}) and (\ref{hyperbola}) possesses (four) solutions for $a$ and $b$. Thus, $ [0,1] \subseteq W_P(T) $, as expected. 
 
\end{proof}

Let $\mathbb{H}$ be a complex Hilbert space and let $T\in \mathbb{L}(\mathbb{H}).$ It follows from Theorem 2.2 of \cite{SP} that if $M_T\neq\emptyset$, then $M_T=S_{\mathbb{H}_0}$, for some subspace $\mathbb{H}_0$ of $\mathbb{H}.$ Therefore, $T$ is a scalar multiple of an isometry on $\mathbb{H}_0$, and we have the following:

\begin{cor}\label{Generalization of Toeplitz-Hausdorff Theorem M}
Let $ \mathbb{H} $ be a complex Hilbert space and let $ T,A\in \mathbb{L}(\mathbb{H}) $. Then
either $M_T=\emptyset$ or $\left\{ \langle Ax,Tx \rangle :~ x\in M_T \right\}$
is a  convex subset of $\mathbb{C}.$
\end{cor}

The well-known Toeplitz-Hausdorff Theorem is an immediate consequence of Theorem \ref{Generalization of Toeplitz-Hausdorff Theorem}. 

\begin{cor}(Toeplitz-Hausdorff Theorem)\label{Toeplitz Hausdorff Theorem}
Let $ \mathbb{H} $ be a complex Hilbert space and let $ A\in \mathbb{L}(\mathbb{H}) $. Then the numerical range of $ A $, defined by
\[ W(A) := \{ \langle Ax, x \rangle :~ x\in S_{\mathbb{H}} \} \]
is a convex subset of $\mathbb{C}$.
\end{cor}

We finish the article with an elementary proof of the Bhatia-\v{S}emrl Theorem that follows as an application of Theorem \ref{Relation between direction and M_T} and Corollary \ref{Generalization of Toeplitz-Hausdorff Theorem M}.\\

\noindent {\emph{(Bhatia-\v{S}emrl Theorem)
Let $ \mathbb{H} $ be a finite-dimensional Hilbert space. Let $ T, A \in \mathbb{L}(\mathbb{H}) $. Then $ T \perp_B A $ if and only if there exists $ x\in M_T $ such that $ \langle Tx, Ax \rangle = 0.$}}\\

\begin{proof}
The verification of sufficiency is straightforward; hence we omit it. Now, assume that $ \langle Tx, Ax \rangle \neq 0$ for all $x\in M_T.$ In other words, $0\notin W_A(T):=\left\{ \langle Ax,Tx \rangle :~ x\in M_T \right\}.$ Since $\mathbb{H}$ is finite-dimensional, $W_A(T)$ is a compact convex subset of $\mathbb{C}$ (Corollary \ref{Generalization of Toeplitz-Hausdorff Theorem M}). Therefore, $W_A(T)$ possesses a unique closest point from $ 0 $, say $ \mu. $ Let $\kappa = \dfrac{\mu}{|\mu|}.$ Therefore, $ \mathsf{Re}~\overline{\kappa} \langle Ax,Tx \rangle \geq |\mu|, $ for all $ x\in M_T $. Thus, by Theorem \ref{Characterization of orthogonality in Hilbert space in a particular direction}, $ Tx \not\perp_{\overline{\kappa}} Ax $ for all $ x\in M_T $. Note that the set $M_T$ is
connected, because by \cite{SP}, it is equal to the (connected) circle $S_{\mathbb{H}_0}$ for some subspace $\mathbb{H}_0$ of $\mathbb{H}.$ Hence by Theorem \ref{Relation between direction and M_T}, $T\not\perp_B A$.
 
\end{proof}

\begin{acknowledgement}
The research of Dr. Debmalya Sain is sponsored by Dr. D. S. Kothari Post-doctoral fellowship. Dr. Sain feels elated to acknowledge the loving friendship of Krystina DeLeon. The research of Mr. Saikat Roy is supported by CSIR MHRD in form of Senior Research Fellowship under the supervision of Prof. Satya Bagchi.\\

The authors are extremely thankful to the anonymous referee for his/her many helpful comments which improved the overall outfit of our paper.
\end{acknowledgement}


\begin{thebibliography}{99}

\bibitem{BS} R. Bhatia, P. \v{S}emrl, \textit{Orthogonality of matrices and distance problems}, Linear Algebra Appl. \textbf{287 }(1999)  77-85.

\bibitem{BP} T. Bhattacharyya, P. Grover, \textit{Characterization of Birkhoff-James orthogonality }, J. Math. Anal. Appl. \textbf{407} (2013) 350-358.

\bibitem{B} G. Birkhoff, \textit{Orthogonality in linear metric spaces}, Duke Math. J. \textbf{1} (1935) 169-172.



\bibitem{GK} K. Gustafson, \textit{The Toeplitz-Hausdorff Theorem for Linear Operators}, Proc. Amer. Math. Soc. \textbf{25} (1970) 203-204.	

\bibitem{J} R. C. James, \textit{Inner product in normed linear spaces}, Bull. Amer. Math. Soc. \textbf{53} (1947) 559-566.

\bibitem{Ja} R. C. James, \textit{Orthogonality and linear functionals in normed linear spaces}, Trans. Amer. Math. Soc. \textbf{61} (1947) 265-292.

\bibitem{K} D. J. Ke\v{c}ki\`{c}, \textit{Gateaux derivative of B(H) norm},  Proc. Amer. Math. Soc. \textbf{133} (2005) 2061-2067.


\bibitem{PSMM} K. Paul, D. Sain, A. Mal, K. Mandal, \textit{Orthogonality  of bounded linear  operators  on complex Banach spaces}, Adv. Oper. Theory. \textbf{3} (2018) 699-709.


\bibitem{S} D. Sain, \textit{Birkhoff-James orthogonality of linear operators on finite dimensional Banach spaces}, J. Math. Anal. Appl. \textbf{447} (2017)  860-866.


\bibitem{SP} D. Sain, K. Paul, \textit{Operator norm attainment and inner product spaces}, Linear Algebra Appl. \textbf{439} (2013) 2448-2452.

\bibitem{SRBB} D. Sain, S. Roy, S. Bagchi, V. Balestro, \textit{A study of symmetric points in Banach spaces}, Linear Multilinear Algebra. (2020).\\ Available from: https://doi.org/10.1080/03081087.2020.1749541.


\bibitem{TA} A. Turn\v{s}ek, \textit{A remark on orthogonality and symmetry of operators in B(H)}, Linear Algebra Appl. \textbf{535} (2017) 141-150.


\end{thebibliography}
\end{document}